\documentclass[conference]{IEEEtran}
\IEEEoverridecommandlockouts
\usepackage{cite}
\usepackage{amsmath,amssymb,amsfonts}
\usepackage{algorithmic}
\usepackage{graphicx}
\usepackage{textcomp}
\usepackage{xcolor}
\usepackage{amsthm}
\usepackage{algorithm,algorithmic}
\newlength\myindent % define a new length \myindent
\newcommand\bindent{%
  \begingroup % starts a group (to keep changes local)
  \setlength{\itemindent}{\myindent} % set itemindent (algorithmic internally uses a list) to the value of \mylength
  \addtolength{\algorithmicindent}{\myindent} % adds \mylength to the default indentation used by algorithmic
}
\newcommand\eindent{\endgroup} % closes a group
\def\BibTeX{{\rm B\kern-.05em{\sc i\kern-.025em b}\kern-.08em
    T\kern-.1667em\lower.7ex\hbox{E}\kern-.125emX}}
\def\grad{\nabla}

\def\bq{\mathbf{q}}
\def\br{\mathbf{r}}
\def\bs{\mathbf{s}}

\def\bu{\mathbf{u}}

\def\bx{\mathbf{x}}  %{\mbox{\boldmath $\lambda$}}
\def\by{\mathbf{y}}
\def\bz{\mathbf{z}}

\def\bC{\mathbf{C}}
\def\bD{\mathbf{D}}

\def\bI{\mathbf{I}}
\def\bJ{\mathbf{J}}

\def\bL{\mathbf{L}}

\def\bV{\mathbf{V}}

\def\cB{\mathcal{B}}

\def\cE{\mathcal{E}}

\def\cG{\mathcal{G}}
\def\cH{\mathcal{H}}
\def\cI{\mathcal{I}}

\def\cL{\mathcal{L}}

\def\cN{\mathcal{N}}
\def\cO{\mathcal{O}}

\def\cS{\mathcal{S}}
\def\cT{\mathcal{T}}

\def\cX{\mathcal{X}}

\def\mE{\mathbb{E}}

\def\smskip{\smallskip}

\def\texitem#1{\par\smskip\noindent\hangindent 25pt
               \hbox to 25pt {\hss #1 ~}\ignorespaces}

\def\abs#1{\left|#1\right|}
\def\norm#1{\left\|#1\right\|}

\newcommand{\BEAS}{\begin{eqnarray*}}
\newcommand{\EEAS}{\end{eqnarray*}}
\newcommand{\BEA}{\begin{eqnarray}}
\newcommand{\EEA}{\end{eqnarray}}
\newcommand{\BEQ}{\begin{eqnarray}}
\newcommand{\EEQ}{\end{eqnarray}}
\newcommand{\BIT}{\begin{itemize}}
\newcommand{\EIT}{\end{itemize}}
\newcommand{\BNUM}{\begin{enumerate}}
\newcommand{\ENUM}{\end{enumerate}}

\newcommand{\BA}{\begin{array}}
\newcommand{\EA}{\end{array}}

\newcommand{\ones}{\mathbf 1}

\newcommand{\reals}{\mathbb{R}}

\newcommand{\diag}{\mathop{\bf diag}}

\newcommand{\dom}{\mathop{\bf dom}}

\newif\ifpagenumbering
\pagenumberingtrue

\pagenumberingfalse

\newsavebox{\theorembox}
\newsavebox{\lemmabox}
\newsavebox{\defnbox}
\newsavebox{\corollarybox}
\newsavebox{\remarkbox}
\newsavebox{\assbox}
\savebox{\theorembox}{\noindent\bf Theorem}
\savebox{\lemmabox}{\noindent\bf Lemma}
\savebox{\defnbox}{\noindent\bf Definition}
\savebox{\corollarybox}{\noindent\bf Corollary}
\savebox{\remarkbox}{\noindent\bf Remark}
\newtheorem{remark}{\usebox{\remarkbox}}[section]
\newtheorem{defn}{\usebox{\defnbox}}

\newtheorem{theorem}{Theorem}[section]
\newtheorem{corollary}{Corollary}[theorem]
\newtheorem{assumption}{Assumption}[section]
\newtheorem{lemma}[theorem]{Lemma}

\usepackage{subfig}
\def\prox#1{\mathbf{prox}_{#1}}
\DeclareMathOperator*{\argmin}{\arg\!\min}

\def\fprod#1{\left\langle#1\right\rangle}
\def\blambda{\boldsymbol{\lambda}}

\begin{document}
\allowdisplaybreaks
\title{An Accelerated Asynchronous Distributed Method for Convex Constrained
Optimization Problems \\
\thanks{The first and third authors would like to acknowledge support from NSF ECCS-2127696.}
}

\author{\IEEEauthorblockN{Nazanin Abolfazli}
\IEEEauthorblockA{\textit{Systems and Industrial Engineering} \\
\textit{University of Arizona}\\
Tucson, AZ 85721 \\
nazaninabolfazli@arizona.edu}
\and
\IEEEauthorblockN{Afrooz Jalilzadeh}
\IEEEauthorblockA{\textit{Systems and Industrial Engineering} \\
\textit{University of Arizona}\\
Tucson, AZ 85721 \\
afrooz@arizona.edu}
\and
\IEEEauthorblockN{Erfan Yazdandoost Hamedani}
\IEEEauthorblockA{\textit{Systems and Industrial Engineering} \\
\textit{University of Arizona}\\
Tucson, AZ 85721 \\
erfany@arizona.edu}
}
\maketitle

\begin{abstract}
We consider a class of multi-agent cooperative consensus optimization problems with local nonlinear convex constraints where only those agents connected by an edge can directly communicate, hence, the optimal consensus decision lies in the intersection of these private sets. We develop an asynchronous distributed accelerated primal-dual algorithm to solve 
the considered problem. The proposed scheme is the first asynchronous method with an optimal convergence guarantee for this class of problems, to the best of our knowledge. In particular, we provide an optimal convergence rate of $\mathcal O(1/K)$ for suboptimality, infeasibility, and consensus violation.
\end{abstract}


\section{Introduction}
Let $\cG=(\cN,\cE)$ denote a connected undirected graph of $N$ computing nodes where $\cN\triangleq \{1,\hdots,N\}$ and $\cE\subseteq \cN\times \cN$ represents the set of edges. We consider the following constrained optimization problem over network $\cG$:
\begin{align}\label{eq:central_problem}
    &\min_{x\in \reals^n} ~\sum_{i\in\cN}\varphi_i(x)\triangleq f_i(x)+\rho_i(x) \\
    &\quad \hbox{s.t.}\quad  g_i(x)\leq 0, ~ i\in\cN,\nonumber
\end{align}
where $x$ denotes the global decision variable; $\rho_i: \mathbb{R}^n \rightarrow \mathbb{R}\cup\{+\infty\}$ is a possibly \emph{non-smooth} convex function with easy-to-compute proximal map; $f_i: \mathbb{R}^n \rightarrow \mathbb{R}$ is a \emph{smooth} convex function; 
%$X\triangleq \bigcap_{i\in\cN} X_i$ where $X_i\subset \reals^n$ is a compact and convex set; 
and $g_i:\reals^n \to \reals^{m_i}$ is a vector-valued convex function. We assume that each agent $i\in\cN$ has only access to local information, i.e., $f_i$, $\rho_i$, and $g_i$. Our objective is to develop an efficient algorithm with a convergence guarantee for solving \eqref{eq:central_problem} in a decentralized fashion using the computing nodes $\cN$ and exchanging information only along the edges $\cE$.

Decentralized optimization over communication networks
has various applications. Here we discuss a few applications. 
%(i) distributed parameter estimation in wireless sensor networks; (ii) multi-agent localization in sensor networks; (iii) processing distributed big data in machine learning; and (iv) power control problems in cellular
% networks -- see \cite{nedic2018distributed,assran2020advances} and the references therein for more details. 
1) In \emph{multi-agent control design},
%a problem of interest is to come up with
consider computing an optimal consensus decision satisfying the constraint of each agent $i\in\cN$, involving some {\it uncertain parameter} $q_i\in\reals^n$, with at least $1-\epsilon$ probability, i.e., %Indeed, %this problem can be formulated as
%consider %the following formulation,
$\min_x\{\sum_{i\in\cN} f_i(x)\mid\mathbb{P}(\{q_i: q_i^\top x\leq b_i\})\geq 1-\epsilon, i\in\cN\}$. %where $q_i$ is the uncertain parameter.
This problem can be formulated as a minimization %of sum of cost functions subject to
over an intersection of ellipsoids under a particular %probability
distribution of the uncertain parameters  \cite{lagoa2005probabilistically}. 2) In \emph{multi-agent localization in sensor networks} one needs to collaboratively locate a target $\bar{x}\in\reals^n$. Suppose each agent $i\in\cN$ has a directional sensor that can detect a target when it belongs to $\cX_i=\{y_i:~\norm{A_iy_i-b_i}_2\leq \eta_i\}$. Let $\cI=\{i\in\cN:~\bar{x}\in \cX_i\}$. Therefore, the target location $\bar{x}$ can be estimated by solving $\min_x\{\norm{x}_2:~x\in\cap_{i\in\cI} \cX_i\}$. %can collect information according to its reachable range which can be characterized as an ellipsoid.
%3) \rev{\emph{Rendez-vous point problem} arises in multi-agent control systems where the objective is to compute a common point in the reachability sets of all agents~\cite{stewart2010cooperative}. Consider a  linear system $x_i(t+1)=A_i x_i(t)+B_iu_i(t)$ for each agent with \saa{a} finite horizon $T$. Suppose the control variable $u_i(t)$ %has a constraint
%must satisfy $\norm{u_i(t)}\leq r_i$ for $i\in\cN$ and $t=0,\hdots,T-1$. Since $x_i(T)$ is an affine map of $\{u_i(t)\}_{t=0}^{T-1}$, $x_i(T)$ lies in an ellipsoid $\cE_i(T)$; hence, %again the objective is
%agents can meet at time $T$ if they find a point in $\cap_{i\in\cN}\cE_i(T)$. 
3) In \emph{distributed robust optimization}, the goal is to %provide a robust solution to a distributed optimization problem under some uncertainty. One of the generic form of this problem can be formulated as
solve $\min_{x\in X}\{\sum_{i\in\cN} f_i(x)\mid g_i(x,q)\leq 0, ~\forall q\in Q,~ \forall i\in\cN\}$ where $Q$ represents an uncertainty set. Using a scenario-based approach~\cite{you2018distributed}, this problem can be reformulated as \eqref{eq:central_problem}. 

Recently, there have been many studies on developing distributed algorithms for convex constrained consensus optimization problems subject to (non)linear constraints with a convergence rate guarantee \cite{chang2014distributed,aybat2016primal,hamedani2021decentralized}. 
However, the update of such algorithms is in a synchronous fashion which requires access to a global clock, thus largely limiting their applicability. Due to the lack of such an assumption, i.e, an agent has to work based on
its own clock, the development of asynchronous algorithms is of prime importance. Additionally, in many applications, networks are vulnerable to certain possible link failures and some agents may not implement any operation at a
certain time instant. Therefore, asynchronous implementation is crucial in process of communication and
computation. 
%However, there has not been any  asynchronous distributed algorithm for convex-constrained optimization problems with an optimal convergence guarantee to the best of our knowledge. 
Next, we briefly summarize the related research, and then we state our contributions. 

\subsection{Literature Review}
In the past few years, numerous research studies have been conducted concerning distributed optimization methods. %Distributed optimization problems, 
In the absence of constraints, various methods under both static and time-varying communication network have been studied, such as \cite{jakovetic2013convergence,nedic2014distributed,shi2015extra} to name a few. 
% In the majority of these studies, the underlying network is allowed to be time varying. To mention a few, a Newton-Raphson consensus strategy is provided
% in \cite{zanella2012asynchronous} for solving unconstrained, convex optimization problems under asynchronous, symmetric gossip communications. For solving smooth, unconstrained optimization problems authors in \cite{shi2015extra} presented  a primal, synchronous methodology named EXTRA.  Accelerated distributed gradient methods handling unconstrained optimization problems over symmetric, time varying networks is presented in \cite{jakovetic2013convergence}. To address
% time-varying and directed graph topologies, a push-sum algorithm for average consensus coupled with a primal subgradient method is proposed in \cite{nedic2014distributed} for solving unconstrained convex optimization problems. 
%It is better to address that the benefit of primal type approaches is that they require low computational complexity. However, the need for decreasing step sizes makes this class of methods to have slow convergence or less accuracy \cite{ling2015dlm}.
For convex optimization problems with easy-to-project local constraints, in \cite{lee2013distributed} authors introduced a distributed random projection algorithm, that can be employed by multiple agents connected over a 
%(balanced)
time-varying network, while a proximal minimization perspective is proposed in \cite{margellos2017distributed}. In a time-varying setting, \cite{nedic2010constrained} suggested a projected subgradient method to solve distributed convex optimization problems. 
%in the presence of constraints.
% It should be noted that in \cite{nedic2010constrained} the particular case where all the agents' constraints are identical is also taken into account. 
For minimizing multi-agent convex optimization problems with linearly coupled constraints over networks, the author in \cite{necoara2013random} proposed (primal) randomized block-coordinate descent methods.
Moreover, there have been several methods proposed  considering a nonconvex objective function subject to easy-to-project or linear constraints such as \cite{bianchi2012convergence,di2016next,sun2016distributed,wai2016projection,hong2017prox}. 

%\cite{zhu2011distributed}

Recently primal-dual methods have become a popular approach for solving distributed optimization problems (see e.g.,   \cite{chang2014distributed,shi2018augmented}).
%yuan2015regularized
In particular, Alternating Directions Method of Multipliers (ADMM) is the basis of many effective distributed algorithms such as those presented in \cite{chang2014multi,makhdoumi2017convergence,ling2015dlm}. Due to the fact that the primal variables must achieve successive minimizations at each iteration, distributed ADMM is computationally expensive \cite{ling2015dlm}.  To address this issue, the linearized ADMM algorithms have been presented in \cite{ling2015dlm} and \cite{aybat2017distributed}. %and saddle point methods
%since only one gradient-like step is required to update the primal variables at each iteration.

Although the conventional implementation of distributed algorithms requires synchronous communication between agents, distributed systems may commonly use asynchronous communication networks between nodes. 
 % as asynchronous communication protocols are a better choice for real-world networks \cite{bertsekas2015parallel}. 
Various asynchronous versions of distributed optimization algorithms have been studied in the literature \cite{jakovetic2011cooperative,iutzeler2015explicit,notarnicola2016asynchronous,latafat2018plug,xiao2019dscovr} with a convergence rate guarantee focusing on easy-to-project or linear constraint sets. However, there are far fewer studies considering general nonlinear constraints. In particular, very recently there have been few studies focusing on developing ADMM-based methods for solving \eqref{eq:central_problem} (with possibly nonconvex functions) such as \cite{farina2019distributed,li2020distributed,tang2022fast} where the agent's updates are asynchronous. However, none of these methods present a convergence rate guarantee for their algorithms. %\cite{nedic2010asynchronous,iutzeler2013asynchronous} based on
Next, we outline the contributions of our paper.
\section{Contributions}
We consider a class of distributed optimization subject to agent-specific nonlinear constraints. We propose an accelerated primal-dual algorithm with asynchronous updates where each agent has
to work based on its own clock and does not need to have access to the global clock.
By assuming a composite convex structure on the primal functions and convex constraints, we show that our proposed algorithm converges to an optimal solution at a rate of $\mathcal O(1/K)$ in terms of suboptimality, infeasibility, and consensus violation.
To the best of our knowledge, the proposed scheme is the first asynchronous method with a momentum acceleration that achieves the optimal convergence rate for the considered setting.

% We propose an asynchronous accelerated primal-dual algorithms for distributed optimization subject to nonlinear constraints. In particular, agents are not required to know any global parameter and they only know who their neighbors are and each agent has
% to work based on its own clock and do not need to have access to the global clock. The proposed scheme is the first asynchronous method with momentum term that achieves the optimal convergence rate of $\mathcal O(1/K)$ for suboptimality, infeasibility, and consensus violation.

{\bf Organization of the paper.} In Section \ref{pre}, we provide the main assumptions and definitions, required for the convergence
analysis. Next, in Sections \ref{method} and \ref{conv rate}, we introduce AD-APD method and show the convergence rate of $\mathcal O(1/K)$ for both suboptimality and infeasibility. Finally, in Section \ref{numeric} we compare the performance of the proposed algorithm with a competitive scheme.  

\section{Preliminaries}\label{pre}
Consider problem \eqref{eq:central_problem}, where $f_i,g_i$ and $\rho_i$ satisfy the following assumption for any $i\in \mathcal N$.
\begin{assumption}\label{assump1}
For each $i\in\cN$, (i) $f_i$ is differentiable on an open set containing $\dom(\rho_i)$ with a Lipschitz continuous gradient $\grad f_i$ and Lipschitz constant $L^f_i$. (ii) $g_i$ is differentiable with Lipschitz continuous Jacobian matrix $\bJ g_i\in\reals^{n\times m_i}$ with constant $L^g_i$. (iii) $\dom(\rho_i)\triangleq \{x\in\reals^n\mid \rho_i<\infty\}$ is bounded. 
\end{assumption}
Before discussing our communication network and  the related assumptions, first, we introduce important notations.% that we use throughout the paper.

\subsection{Notations} Throughout the paper, $\norm{.}$ denotes the Euclidean norm. 
%Given a convex set $\cS$, %let $\sigma_{\cS}(.)$ denote its support function, i.e., $\sigma_{\cS}(\theta)\triangleq\sup_{w\in \cS}\fprod{\theta,~w}$, 
%let $\mI_{S}(\cdot)$ denote the indicator function of $\cS$, i.e., $\mI_{S}(w)=0$ for $w\in\cS$ and equal to $+\infty$ otherwise, and 
Given a convex function $f$, let $\mbox{prox}_{f}(w)\triangleq\argmin_v\{f(v)+\frac{1}{2}\norm{v-w}^2\}$ denote the proximal operator of function $f$. 
%For a closed convex set $\cS$, we define the distance function as $d_{\cS}(w)\triangleq\norm{\cP_{\cS}(w)-w}$.
Let $\bI_n\in\reals^{n\times n}$ denote the identity matrix and $\ones_n\in\reals^n$ denote the vector of ones. Let $\otimes$ denote Kronecker product and $[x]_+\triangleq \max\{0,x\}$. For any matrix $A\in\reals^{n\times m}$, $A_{i:}\in\reals^{1\times m}$ and $A_{:j}\in\reals^{n\times 1}$ denotes $i$-th row and $j$-th column of $A$, respectively, and $a_{ij}$ denotes row $i$ and column $j$ of matrix $A$. For any set of vectors $\{x_i\}_{i\in\cN}\subset\reals^n$, $\bx=[x_i]_{i\in\cN}\in\reals^{nN}$ denotes the concatenation of those vectors. Moreover, for a given set of matrices $A_i\in \reals^{n_i\times m_i}$, for $i\in\cN$, $\diag([A_i]_{i\in\cN})\in\reals^{n\times m}$ denotes a block diagonal matrix whose diagonal blocks are $A_i$'s where $(n,m)=\sum_{i\in\cN}(n_i,m_i)$. 

\begin{remark}\label{rem1}
Based on Assumption \ref{assump1}, the boundedness of the domain implies that for any $i\in\cN$, $g_i$ is a Lipschitz continuous function and we denote the constant with $C_i$.
\end{remark}
Next, we define some notations based on the constants introduced in Assumption \ref{assump1} and Remark \ref{rem1}.
\begin{defn}\label{def1}
Given a set of parameters $\tau_i$, $\gamma_i$, and $\sigma_i$ for $i\in\cN$, let $\cT\triangleq \diag([\frac{1}{\tau_i}\bI_n]_{i\in\cN})$, $\cS\triangleq \diag([\frac{1}{\sigma_i}\bI_{m_i}]_{i\in\cN})$, $\Gamma\triangleq \diag([\frac{1}{\gamma_i}\bI_n]_{i\in\cN})$, and $\cB\triangleq \diag(\cS,\Gamma)$. Moreover, we define $\bC\triangleq \diag([C_i\bI_{m_i}]_{i\in\cN})$, $\bD\triangleq \diag([(C_i+\delta_i)\bI_{n}]_{i\in\cN})$, and $\Delta\triangleq \diag([\delta_i\bI_n]_{i\in\cN})$. 
\end{defn}
\begin{defn}\label{def2}
Let $\varphi(\bx)\triangleq \sum_{i\in\cN}\varphi_i(x_i):\reals^{nN}\to\reals$ and $g(\bx)\triangleq [g_i(x_i)]_{i\in\cN}:\reals^{nN}\to\reals^m$ where $m\triangleq \sum_{i\in\cN}m_i$.
\end{defn}

\subsection{Communication network}
We consider a multi-agent system where the agents combine  their own information state with those received from their neighbors to update their state. Suppose nodes $i$ and $j$ can exchange information only if $(i,j) \in \cE$ or $(j,i) \in \cE$, and each node $i\in\cN$ has a \emph{private} (local) cost function $\varphi_i$ and constraint function $g_i$. The set of neighboring nodes of agent $i$ is denoted by $\cN_i \triangleq \{j\in\cN\mid (i,j)\in \cE \hbox{ or } (j,i)\in\cE\}$. The weighted matrix $W = [w_{ij}] \in \reals^{N\times N}$ is a nonnegative matrix such that $w_{ij} > 0$ if  $j\in \cN_i$ and $w_{ij} = 0$ otherwise. 

We assume that each node $i\in\cN$ has a local clock $t_i\in\reals_+$ and a randomly generated waiting time $T_i$. Each node $i$ will remain \emph{idle} while $\tau_i<T_i$ and switches to the \emph{awake} mode when $\tau_i=T_i$ after which it runs the local computations, resets $t_i=0$ and draws a new realization of the random variable $T_i$. Formally, we make the following assumptions on the communication architecture.
\begin{assumption}\label{assump2}
The waiting times $T_i$ between consecutive  events are i.i.d. random variables with the same exponential distribution. Moreover, only one node can be awake at each time instant.
\end{assumption}
%We consider a multi-agent system where the agents combine  their own information state with those received from their neighbors to update their state. In particular, we assume that there exists a symmetric and nonnegative doubly-stochastic weight matrix $W\in\reals^{N\times N}$ compliant with the graph $\cG$, such that if $(i,j)\in\cE$ then $w_{ij}>0$, otherwise $w_{ij}=0$.  

\subsection{Problem Reformulation}
Let $x_i\in\reals^n$ denote the local decision vector of node $i\in\cN$. We can reformulate \eqref{eq:central_problem} as $\min_{\bx}\{\varphi(\bx)\mid g(\bx)\leq 0,~x_i=x_j~\forall(i,j)\in\cE\}$. Furthermore, we can describe the consensus constraint as a linear constraint, i.e., $(V\otimes\bI_n)\bx=\mathbf{0}$ for some $V\in\reals^{N\times N}$. We consider the following condition on the consensus constraint matrix $V$.
\begin{assumption}\label{assum:consensus}
For any $\bx\in\reals^{nN}$, $(V\otimes\bI_n)\bx=\mathbf{0}$ if and only if there exists $x\in\reals^n$ such that $\bx=\ones_N\otimes x$.
Moreover, for any $i\in\cN$, there exists $\delta_i>0$ such that $\norm{V_{i:}}_1\leq \delta_i$. 
\end{assumption}
\begin{remark}
%Here we describe a few instances of matrix $V$ satisfying Assumption \ref{assum:consensus}. 
For any mixing matrix $W$ with eigenvalues in $(-1,1]$, e.g., Laplacian-based and Metropolis mixing matrices, let $V=\alpha(\bI_N-W)$ for any $\alpha>0$. It implies that Assumption \ref{assum:consensus} is satisfied with $\delta_i=2\alpha(1-w_{ii})$. 
\end{remark}

Furthermore, we consider the following standard regularity assumption on problem \ref{eq:central_problem}.
\begin{assumption}\label{assum:dual-bound}
The duality gap for \eqref{eq:central_problem} is zero, and a primal-dual solution $(x^*,\by^*)$ to \eqref{eq:central_problem} exists. Moreover, the dual solution is bounded, i.e., $\exists B>0$ such that $\norm{\by^*}\leq B$.
\end{assumption}
\begin{remark}
Note that Assumption \ref{assum:dual-bound} holds in practice under mild conditions as it is studied in \cite{aybat2019distributed}. For instance, when Slater condition holds the agents can collectively compute a Slater point and use it to find a dual bound in a distributed manner.
\end{remark}

Now using Lagrangian duality, we equivalently write the following saddle point formulation. 
\begin{align}\label{eq:SP} \min_{\bx\in\reals^{nN}}\max_{\substack{\by\in\reals^{m}_+\\ \blambda\in\reals^{nN}}}\cL(\bx,\by,\blambda)\triangleq \varphi(\bx)+\fprod{g(\bx),\by}+\fprod{\blambda,\bV\bx}
\end{align}
where $\bV\triangleq V\otimes\bI_n$. Next, we develop a distributed primal-dual method with convergence guarantee for solving \eqref{eq:SP}.

\section{Proposed Method}\label{method}
In this section, we study the asynchronous distributed implementation of the accelerated primal-dual (APD) algorithm to solve \eqref{eq:SP}. We propose an asynchronous distributed accelerated primal-dual (AD-APD) algorithm whose iterations can be computed in a decentralized way, via the node-specific computations as in Algorithm \ref{alg:AD-APD}.
In particular, at each iteration, one agent goes to "awake" mode uniformly at random and updates its local decision variables by taking dual accent steps with momentum accelerations following a proximal-gradient descent using the most updated dual decision variables. Moreover, each agent combines the local information with its neighbors using the consensus constraint matrix. Moreover, our proposed method includes a new linear combination of dual gradient iterates that can recover APD for $N=1$.  
%we propose AD-APD method, displayed in Algorithm \ref{alg:AD-APD} to solve problem \eqref{eq:central_problem}.

\begin{algorithm}
    \caption{Asynchronous Distributed Accelerated Primal-Dual Algorithm (AD-APD)}
    \label{alg:AD-APD}
    \begin{algorithmic}
    \STATE\hspace*{-3mm}{\bf Input:} $[\tau_i,\sigma_i,\gamma_i]_{i\in\cN}$, $(\bx^0,\by^0,\blambda^0)\in\reals^{nN}\times\reals^m\times\reals^{nN}$\\[2mm]
    \STATE\hspace*{-3mm}For $k\geq 0$,\\[1mm]
    \STATE\hspace*{-3mm}\texttt{IDLE}:\\[1mm]
    \bindent
     \WHILE{ $t_i<T_i$}
        \STATE Do Nothing
     \ENDWHILE
     \STATE Go to \texttt{AWAKE}
      \eindent
      \STATE\hspace*{-3mm}\texttt{AWAKE}:\\[1mm]
      \bindent
      \STATE Receive $\lambda_j^k,x_j^k,x_j^{k-1}$ from neighbors ($j\in\cN_i$)\\[1mm]
      \STATE $y_i^{k+1}\gets\max\big\{\mathbf{0},y_i^k+2N\sigma_i\big(g_i(x_i^k)-\tfrac{(2N-1)}{2N}g_i(x_i^{k-1})\big)\big\}$\\[1mm]
      \STATE $\lambda_i^{k+1}\gets\lambda_i^k+\gamma_i\sum_{j\in\cN_i\cup\{i\}}v_{ij}(2Nx_j^k-(2N-1)x_j^{k-1})$\\[1mm]
      \STATE $x_i^{k+1}\gets \mbox{prox}_{\tau_if_i}\big(x_i^k-\tau_i\big(\bJ g_i(x_i^k)^\top y_i^{k+1}+v_{ii}\lambda_i^{k+1}$\\[1mm] $\qquad\qquad\qquad\qquad+\sum_{j\in\cN_i}v_{ij}\lambda_j^k\big)\big)$\\[1mm]
      \STATE Send $\lambda_i^{k+1},x_i^{k+1},x_i^k$ to neighbors\\[1mm]
      \STATE Set $t_i=0$, get new realization $T_i$ and go to \texttt{IDLE}
      \eindent
     \end{algorithmic}
\end{algorithm}

\section{Convergence Analysis}\label{conv rate}
In the following theorem, we state the convergence rate of AD-APD in terms of the Lagrangian error metric, and then in Corollary \ref{main cor}, the convergence rate in terms of the suboptimality, infeasibility, and consensus violation is shown.

\begin{theorem} \label{main theorem}
    Suppose Assumptions \ref{assump1} - \ref{assum:dual-bound} hold and $\{\bx^k,\by^k,\blambda^k\}_{k\geq 0}$ is the sequence generated by AD-APD stated in Algorithm \ref{alg:AD-APD} with step-sizes selected such that  $\tau_i\leq \tfrac{1}{2(C_i+\delta_i)+L_i^f+B L_i^g}$, $\sigma_i\leq \tfrac{1}{3C_i}$ and $\gamma_i\leq \tfrac{1}{3\delta_i}$. Then it holds for any $(\bx,\by,\blambda)\in\reals^{nN}\times\reals^m_+\times\reals^{nN}$ and $K\geq 1$ that
    \begin{align}
        &\nonumber\mE\big[\cL(\bar\bx^K,\by,\blambda)-\cL(\bx,\bar\by^K,\bar\blambda^K)\big]\\
        &\nonumber\leq \frac{N}{2(K+N-1)}\Big(\norm{\bx^0-\bx}_{\cT+\bD}^2+\norm{\by^0-\by}_{\cS+\bC}^2\\
        &\quad \nonumber+\norm{\blambda^0-\blambda}_{\Gamma+\Delta}^2+\tfrac{N-1}{N}(\cL(\bx^0,\by,\blambda)-\cL(\bx,\by^0,\blambda^0))\Big),
    \end{align}
    where $(\bar\bx^K,\bar\by^K,\bar\blambda^K)\triangleq \frac{1}{K+N-1}\big(\sum\limits_{k=1}^{K-1}(\bx^k,\by^k,\blambda^k)+N(\bx^K,\by^K,\blambda^K)\big)$.
\end{theorem}
\begin{proof}
    The proof is presented in section \ref{sec:proof-thm}. 
\end{proof}
\begin{corollary}\label{main cor}
    Under premises of Theorem \ref{main theorem}, for any $K\geq 1$, the following holds.
    \begin{align*}
        &\abs{\varphi(\bar\bx^K)-\varphi(\bx^*)}\leq \cO\left(\tfrac{N}{K+N-1}\right), \nonumber\\
        &\norm{\blambda^*}\norm{V\bar\bx^K}+\sum_{i\in\cN}\norm{y_i^*}\norm{[g_i(\bar\bx^K)]_+}\leq \cO\left(\tfrac{N}{K+N-1}\right).
    \end{align*}
\end{corollary}
\begin{proof}
The proof follows the same steps as in \cite[Corollary 4.2.]{hamedani2021}.
\end{proof}
Before proving Theorem \ref{main theorem}, we state a standard technical lemma for the proximal gradient step which is a trivial extension of Property 1 in \cite{tseng2008accelerated}. Then we provide
a one-step analysis of the algorithm in Lemma \ref{lem:one-step}.
\begin{lemma}
\label{lem_app:prox}
Let $f:\mathbb R^n\rightarrow\reals$ be a closed convex function. Given $\bar{x}\in\dom f$ and $t>0$, let
\begin{align*}
x^+=\argmin_{x\in\reals^n} f(x)+\tfrac{t}{2}\|x-\bar{x}\|^2.
\end{align*}
Then for all $x\in\reals^n$, the following inequality holds:
\begin{eqnarray*}
f(x)+\tfrac{t}{2}\|x-\bar{x}\|^2\geq f(x^+) + \tfrac{t}{2}\|x^+-\bar{x}\|^2+\tfrac{t}{2}\|x-x^+\|^2. %\label{eq_app:bregman}
\end{eqnarray*}
\end{lemma}

Before we proceed, we define some notations to facilitate the proof. 
\begin{defn}
Let function $\Phi$ be the smooth part of the objective function in \eqref{eq:SP}, i.e., $\Phi_i(x_i,y_i,\blambda)\triangleq f_i(x_i)+\fprod{g_i(x_i),y_i}+\fprod{(V_{i:}\otimes\bI_n)\blambda,x_i}$ and $\Phi(\bx,\by,\blambda)\triangleq \sum_{i\in\cN}\Phi_i(x_i,y_i,\blambda)$. Moreover, we define $\bz\triangleq [\by^\top~\blambda^\top]^\top$. 
\end{defn}

\begin{lemma}\label{lem:one-step}
Let $\{\bx^k,\by^k,\blambda^k\}_{k\geq 0}$ be the sequence generated by AD-APD, stated in Algorithm \ref{alg:AD-APD}. Suppose Assumptions \ref{assump1} and \ref{assump2} hold and $\cT$ and $\cB$ are defined in Definition \ref{def1}. Let $\bz\triangleq [\by^\top,\blambda^\top]^\top$, then for any $(\bx,\by,\blambda)\in\reals^{nN}\times\reals^m_+\times\reals^{nN}$,
\begin{align}\label{eq:one-step}
&\mE^k[\cL(\bx^{k+1},\by,\blambda)-\cL(\bx,\by^{k+1},\blambda^{k+1})]\leq (N-1)\big(\cH^k\nonumber\\
&-\mE^k[\cH^{k+1}]\big)+\fprod{\bu^k,\bz^k-\bz}-\mE^k[\fprod{\bu^{k+1},\bz^{k+1}-\bz}]\nonumber\\
& +\tfrac{N}{2}\mE^k\big[\norm{\bx-\bx^k}_{\cT}^2-\norm{\bx- \bx^{k+1}}_{\cT}^2-\norm{\bx^k-\bx^{k+1}}_{\cT-\bL^\Phi}^2\nonumber\\
&+\norm{\bx^{k-1}-\bx^k}_{2\bD}^2\big]+\tfrac{N}{2}\mE^k\big[\norm{\bz-\bz^k}_{\cB}^2\nonumber\\
&-\norm{\bz-\bz^{k+1}}_{\cB}^2-\norm{\bz^k- \bz^{k+1}}_{\widetilde\cB}^2\big],
\end{align}
where $\cH^k\triangleq \rho(\bx^k)+\Phi(\bx^k,\by^k,\blambda^k)$, $\widetilde\cB\triangleq \cB-2\bC_2/N$,  %$\widetilde\cT\triangleq \cT-\bL^\Phi-2\bD$, 
$\bD\triangleq \diag([(C_i+\delta_i)\bI_n]_{i\in\cN})$, $\bC_2\triangleq \diag([C_i\bI_{m_i}]_{i\in\cN},[\delta_i\bI_{n}]_{i\in\cN})$, and $\bu^k\triangleq \grad_\bz\Phi(\bx^k,\by^k,\blambda^k)-(2N-1)\grad_\bz\Phi(\bx^{k-1},\by^{k-1},\blambda^{k-1})$.
\end{lemma}
\proof
We begin the proof by defining auxiliary sequences $\{\tilde\bx^k,\tilde\blambda^k\}_{k\geq 1}\subseteq\reals^{nN}$ and $\{\tilde\by^k\}_{k\geq 1}\subseteq\reals^m$ representing centralized updates and compare them with the sequences generated by the proposed algorithm. Note that these auxiliary sequences are never actually computed in the implementation of the algorithm. In particular, we define the following for all $i\in\cN$
\begin{align}
&\tilde{y}_i^{k+1} \triangleq \max\{\mathbf{0},y_i^k+\sigma_i(\grad_{y_i}\Phi_i(x_i^k,y_i^k,\blambda^k)+s_i^k)\},
\label{eq:tilde-problem-y}\\
&\tilde\lambda_i^{k+1} \triangleq \lambda_i^k+\gamma_i(\grad_{\lambda_i}\Phi(\bx^k,\by^k,\blambda^k)+r_i^k),
%\argmin_{y_j} \left\{ h_j(y_j)-\fprod{s_j^k, y_j}+\tfrac{1}{2\sigma_j^k}\norm{y_j-y_j^k}^2 \right\},
\label{eq:tilde-problem-lam} \\
&\tilde{x}_i^{k+1} \triangleq \prox{\tau_i\rho_i}(x_i^k-\tau_i\grad_{x_i}\Phi_i(x_i^k,y_i^{k+1},\blambda^{k+1})),
%\argmin_{x_i} \left\{f_i(x_i)+\fprod{\grad_{x_i}\Phi_i(x_i^k,y^{k+1}), ~x_i}+{\tfrac{1}{2\tau_i^k}}\norm{x_i-x_i^k}^2\right\},  
\label{eq:tilde-problem-x}
\end{align}
where $s_i^k\triangleq (2N-1)(\grad_{y_i}\Phi_i(x_i^k,y_i^k,\blambda^k)-\grad_{y_i}\Phi_i(x_i^{k-1},y_i^{k-1},\blambda^{k-1}))$ and $r_i^k\triangleq (2N-1)(\grad_{\lambda_i}\Phi(\bx^k,\by^k,\blambda^k)-\grad_{\lambda_i}\Phi(\bx^{k-1},\by^{k-1},\blambda^{k-1}))$. Applying Lemma~\ref{lem_app:prox} on  \eqref{eq:tilde-problem-x} implies that
\begin{align}\label{eq:x-ineq}
&\rho_i(\tilde x_i^{k+1})-\rho_i(x_i)\leq \fprod{\grad_{x_i}\Phi_i(x_i^k,y_i^{k+1},\blambda^{k+1}),x_i-\tilde x_i^{k+1}} \nonumber\\
&+\tfrac{1}{2\tau_i}\big[\norm{x_i-x_i^k}^2-\norm{x_i-\tilde x_i^{k+1}}^2-\norm{x_i^k-\tilde x_i^{k+1}}^2\big].
\end{align}
Using Lipschitz continuity of $\grad f_i$ and $\bJ g_i$ and boundedness of sequence $\{\by^k\}_{k\geq 0}$ we conclude that $\grad_{x_i}\Phi_i(x_i^k,y_i^{k+1},\blambda^{k+1})$ is Lipschitz continuous with constant $L^\Phi_i= L^f_i+BL^g_i$. Therefore, 
\begin{align}\label{eq:Lip-x}
&\fprod{\grad_{x_i}\Phi_i(x_i^k,y_i^{k+1},\blambda^{k+1}),x_i-\tilde x_i^{k+1}}\leq \Phi_i(x_i,y_i^{k+1},\blambda^{k+1})\nonumber\\
&\quad -\Phi_i(\tilde x_i^{k+1},y_i^{k+1},\blambda^{k+1})+\tfrac{L^\Phi_i}{2}\norm{\tilde x_i^{k+1}-x_i^k}^2.
\end{align}
Combining \eqref{eq:x-ineq} and \eqref{eq:Lip-x}, and summing over $i\in\cN$ we obtain
\begin{align*}%\label{eq:x-ineq-tild}
&\rho(\tilde\bx^{k+1})-\rho(\bx)\leq \Phi(\bx,\by^{k+1},\blambda^{k+1})-\Phi(\tilde\bx^{k+1},\by^{k+1},\blambda^{k+1})\nonumber\\
&+\tfrac{1}{2}\big[\norm{\bx-\bx^k}_{\cT}^2-\norm{\bx-\tilde \bx^{k+1}}_{\cT}^2-\norm{\bx^k-\tilde \bx^{k+1}}_{\cT-\bL^\Phi}^2\big],
\end{align*}
where $\bL^\Phi\triangleq \diag([L_i^\Phi\bI_n]_{i\in\cN})$. Note that at each iteration of the algorithm only one agent is awake, i.e., one component of each decision variable is updated, therefore, for any function $\psi:\reals^{nN}\to\reals$ we have $\mE^k[\psi(\bx^{k+1})]=\frac{1}{N}\psi(\tilde\bx^{k+1})+(1-\frac{1}{N})\psi(\bx^k)$ and one can deduce similar results for $\by^{k+1}$ and $\blambda^{k+1}$. Now using this fact and concavity and smoothness of $\Phi(\bx,\cdot,\cdot)$ from the last inequality we can conclude that
\begin{align}\label{eq:x-ineq-one-step}
&\mE^k[\rho(\bx^{k+1})+\Phi(\bx^{k+1},\by,\blambda)-\rho(\bx)-\Phi(\bx,\by^{k+1},\blambda^{k+1})]\leq \nonumber\\
& (N-1)(\rho(\bx^k)+\Phi(\bx^k,\by^k,\blambda^k)-\mE^k[\rho(\bx^{k+1})\nonumber\\&+\Phi(\bx^{k+1},\by^{k+1},\blambda^{k+1})])\nonumber\\
&+\mE^k\left[\fprod{\grad_{\bz}\Phi(\bx^{k+1},\by^{k+1},\blambda^{k+1}),\bz-\bz^{k+1}}\right]+A_1^k\nonumber\\
&+ (N-1)\mE^k\left[\fprod{\grad_\bz\Phi(\bx^k,\by^k,\blambda^k),\bz^{k+1}-\bz^k}\right],
\end{align}
where $A_1^k\triangleq \tfrac{1}{2}\big[\norm{\bx-\bx^k}_{\cT}^2-\norm{\bx-\tilde \bx^{k+1}}_{\cT}^2-\norm{\bx^k-\tilde \bx^{k+1}}_{\cT-\bL^\Phi}^2\big]$. 

Using a similar argument as in \eqref{eq:x-ineq}, one can obtain the following inequalities for the updates in \eqref{eq:tilde-problem-y} and \eqref{eq:tilde-problem-lam}, respectively. Indeed for any $y_i\in\reals^{m_i}_+$,
\begin{align}
%&h_i(\tilde y_i^{k+1})-h_i(y_i)\leq 
&0 \leq \fprod{\grad_{y_i}\Phi_i(x_i^k,y_i^k,\blambda^k)+s_i^k,\tilde y_i^{k+1}-y_i}\nonumber\\
&+\tfrac{1}{2\sigma_i}\big[\norm{y_i-y_i^k}^2-\norm{y_i-\tilde y_i^{k+1}}^2-\norm{\tilde y_i^{k+1}-y_i^k}^2\big] \label{eq:y-prox}\\
&0=\fprod{\grad_{\lambda_i}\Phi(\bx^k,\by^k,\blambda^k)+r_i^k,\tilde \lambda_i^{k+1}-\lambda_i}+\tfrac{1}{2\gamma_i}\big[\|{\lambda_i-\lambda_i^k}\|^2\nonumber\\
&-\|{\lambda_i-\tilde\lambda_i^{k+1}}\|^2-\|{\tilde\lambda_i^{k+1}-\lambda_i^k}\|^2\big]\label{eq:lam-prox}.
\end{align}

Next, we sum \eqref{eq:y-prox} and \eqref{eq:lam-prox} over $i\in\cN$ and add the resulting inequality to \eqref{eq:x-ineq-one-step}. Then, using definition of $\cL(\bx,\by,\blambda)$ and $\cH^k$, and that $\mE^k[h(\by^{k+1})]=\frac{1}{N}h(\tilde\by^{k+1})+(1-\frac{1}{N})h(\by^k)$, we obtain
\begin{align}\label{eq:sum-L}
&\mE^k[\cL(\bx^{k+1},\by,\blambda)-\cL(\bx,\by^{k+1},\blambda^{k+1})]\leq (N-1)[\cH^k\nonumber\\
&-\mE^k[\cH^{k+1}]]+\fprod{\grad_{\bz}\Phi(\bx^{k+1},\by^{k+1},\blambda^{k+1}),\bz-\bz^{k+1}}\nonumber\\
&+\fprod{\grad_{\bz}\Phi(\bx^k,\by^k,\blambda^k)+\bq^k,\tilde\bz^{k+1}-\bz}+A_1^k+A_2^k+A_3^k\nonumber\\
&+(N-1) \mE^k\left[\fprod{\grad_\bz\Phi(\bx^k,\by^k,\blambda^k),\bz^{k+1}-\bz^k}\right],
\end{align}
where $\bq^k\triangleq [{\bs^k}^\top,{\br^k}^\top]^\top$. Moreover, $A_2^k, A_3^k$ are defined similar to $A_1^k$ as the sum of terms containing norm squares in \eqref{eq:y-prox} and \eqref{eq:lam-prox} over $i\in\cN$, respectively. To simplify the notation, we will use $\grad_\bz\Phi^k\triangleq \grad_\bz\Phi(\bx^k,\by^k,\blambda^k)$. One can easily observe that $\bq^k=(2N-1)(\grad_\bz\Phi^k-\grad_\bz\Phi^{k-1})$. Next, we deal with the two inner product terms in \eqref{eq:sum-L} as follows.
\begin{align}\label{eq:inner-telescop}
&\mE^k\Big[\fprod{\grad_{\bz}\Phi^{k+1},\bz-\bz^{k+1}}+\fprod{\grad_\bz\Phi^k+\bq^k,\tilde\bz^{k+1}-\bz}\nonumber\\
&\quad +(N-1) \fprod{\grad_\bz\Phi^k,\bz^{k+1}-\bz^k}\Big]\nonumber\\
&=\mE^k\Big[\fprod{\grad_{\bz}\Phi^{k+1}-(2N-1)\grad_\bz\Phi^k,\bz-\bz^{k+1}}\Big]\nonumber\\
&\quad +\fprod{\bq^k-2(N-1)\grad_{\bz}\Phi^k,\bz^k-\bz}+N\mE^k[\fprod{\bq^k,\bz^{k+1}-\bz^k}]\nonumber\\
&=\fprod{\bu^k,\bz^k-\bz}-\mE^k[\fprod{\bu^{k+1},\bz^{k+1}-\bz}]\nonumber\\
&\quad+N\mE^k[\fprod{\bq^k,\bz^{k+1}-\bz^k}],
\end{align}
where $\bu^k= \bq^k-2(N-1)\grad_\bz\Phi^k$. Note that at each iteration only one agent's decision variables are updated, hence, $\bq^k=q_i^k$ for $i=i_{k-1}$. Therefore, using Young's inequality and Lipschitz continuity of $g_i$ and $\norm{V_{i:}}\leq\delta_i$ we conclude that for $i=i_{k-1}$,
\begin{align*}
&N\fprod{\bq^k,\bz^{k+1}-\bz^k}=N\fprod{q_i^k,z_i^{k+1}-z_i^k}\nonumber\\
&= N(2N-1)\big(\fprod{g_i(x_i^k)-g_i(x_i^{k-1}),y_i^k-y_i^{k-1}}\nonumber\\
&\quad +\fprod{V_{i:}(x_i^k-x_i^{k-1}),\lambda_i^{k+1}-\lambda_i^k}\big)\nonumber\\
&\leq \tfrac{N(2N-1)}{2}\Big((C_i+\delta_i)\norm{x_i^k-x_i^{k-1}}^2 + C_i\norm{y_i^{k+1}-y_i^k}^2\nonumber\\
&\quad +\delta_i\norm{\lambda_i^{k+1}-\lambda_i^k}^2\Big).
\end{align*}
Note that at iteration $k$ only $z^{k+1}_{i_k}$ is updated where $i_k$ is chosen with probability $1/N$. Hence, one can readily observe that $z^{k+1}_{i_{k-1}}\neq z^k_{i_{k-1}}$ with probability $1/N^2$ and $z^{k+1}_{i_{k-1}}= z^k_{i_{k-1}}$ otherwise. 
Therefore, taking conditional expectation from the above inequality imply that 
\begin{align}\label{eq:inner-prod}
&N\mE^k[\fprod{\bq^k,\bz^{k+1}-\bz^k}]\leq \tfrac{2N-1}{2}\mE^k\big[\norm{\bx^k-\bx^{k-1}}_{\bD}^2\big]\nonumber\\
&\quad +\tfrac{2N-1}{2N}\mE^k\big[\norm{\bz^{k+1}-\bz^k}_{\bC_2}^2\big].
\end{align}
Finally, combining \eqref{eq:inner-prod} and \eqref{eq:inner-telescop} with \eqref{eq:sum-L}, and using the fact that $\norm{\tilde\bx^{k+1}-\bx^k}^2=N\mE^k[\norm{\bx^{k+1}-\bx^k}^2]$ the result can be concluded. \qed

\subsection{Proof of Theorem \ref{main theorem}}\label{sec:proof-thm}
Now we are ready to prove the main result. Consider the inequality obtained in \eqref{eq:one-step}. Taking expectations from both sides, using the step-size selection implying $\widetilde\cB\succeq 0$ and $\widetilde\cT\succeq 0$, summing the resulting inequality from $k=0$ to $K-1$ and dividing by $K$ we obtain 
\begin{align}\label{eq:sum-over-k}
&\frac{1}{K}\sum_{k=0}^{K-1}\mE[\cL(\bx^{k+1},\by,\blambda)-\cL(\bx,\by^{k+1},\blambda^{k+1})]\leq \nonumber\\
&(N-1)\big(\cH^0-\mE[\cH^K]\big)+\fprod{\bu^0,\bz^0-\bz}-\mE[\fprod{\bu^{K},\bz^{K}-\bz}]\nonumber\\
& +\tfrac{N}{2}\mE\big[\norm{\bx-\bx^0}_{\cT}^2-\norm{\bx- \bx^{K}}_{\cT}^2-\norm{\bx^{K-1}-\bx^K}^2_{\cT-\bL^\Phi}\big]\nonumber\\
&+\tfrac{N}{2}\mE\big[\norm{\bz-\bz^0}_{\cB}^2-\norm{\bz-\bz^{K}}_{\cB}^2\big].
\end{align}
%where we used Jensen's inequality on the left-hand side. 
We notice that function $\Phi(\bx,\by,\blambda)$ is linear in $(\by,\blambda)$, i.e., $\grad_\bz\Phi(\bx,\by,\blambda)=\grad_\bz\Phi(\bx,\bar\by,\bar\blambda)$ for any $\by,\bar\by,\blambda,\bar\blambda$; therefore, we can show the following relations for any $k\geq 0$,
\begin{align}\label{eq:Hk-equality}
&(N-1)(\cH^k-\cL(\bx,\by,\blambda))+\fprod{\bu^k,\bz^k-\bz}\nonumber\\
&=(N-1)\big(\cL(\bx^k,\by,\blambda) -\cL(\bx,\by,\blambda)\big)+\fprod{\bq^k,\bz^k-\bz}\nonumber\\
&\quad -(N-1)\fprod{\grad_\bz\Phi^k,\bz^k-\bz}\nonumber\\
&=(N-1)\big(\cL(\bx^k,\by,\blambda) -\cL(\bx,\by,\blambda)\big)+\fprod{\bq^k,\bz^k-\bz} \nonumber\\
&\quad -(N-1)\fprod{\grad_\bz\Phi^k,\bz^k-\bz}\nonumber\\
&\quad \pm(N-1) \fprod{\grad_\bz\Phi(\bx,\by^k,\blambda^k),\bz^k-\bz}\nonumber\\
&= (N-1)\big(\cL(\bx^k,\by,\blambda) -\cL(\bx,\by^k,\blambda^k)\big)+\fprod{\bq^k,\bz^k-\bz} \nonumber\\
&\quad +(N-1)\fprod{\grad_\bz\Phi(\bx,\by^k,\blambda^k)-\grad_\bz\Phi^k,\bz^k-\bz}.
\end{align}
Next, we provide upper bounds for the two inner products on the right-hand side of \eqref{eq:Hk-equality} similar to \eqref{eq:inner-prod}.
\begin{align*}
&|\fprod{\bq^k,\bz^k-\bz}|\leq \tfrac{2N-1}{2}\big(\norm{\bx^k-\bx^{k-1}}_{\bD}^2+\norm{\bz^k-\bz}_{\bC_2}^2\big)\\
& (N-1)\Big|\fprod{\grad_\bz\Phi(\bx,\by^k,\blambda^k)-\grad_\bz\Phi^k,\bz^k-\bz}\Big|\leq\nonumber\\
&\quad \tfrac{N-1}{2}\big(\norm{\bx^k-\bx}_{\bD}^2+\norm{\bz^k-\bz}_{\bC_2}^2\big).
\end{align*}
Now, with the help of above inequalities in \eqref{eq:Hk-equality} once for $k=0$ and once for $k=K$, the fact that $\bq^0=0$, and using the resulting inequality within \eqref{eq:sum-over-k} we obtain 
\begin{align*}
&\frac{1}{K}\sum_{k=0}^{K-1}\mE[\cL(\bx^{k+1},\by,\blambda)-\cL(\bx,\by^{k+1},\blambda^{k+1})] \nonumber\\
&\leq (N-1)(\cL(\bx^{0},\by,\blambda)-\cL(\bx,\by^{0},\blambda^{0}))\nonumber\\
&-(N-1)(\cL(\bx^{K},\by,\blambda)-\cL(\bx,\by^{K},\blambda^{K}))\nonumber\\
& +\tfrac{N}{2}\mE\big[\norm{\bx-\bx^0}_{\cT+\bD}^2-\norm{\bx- \bx^{K}}_{\cT-\bD}^2-\norm{\bx^{K-1}-\bx^K}^2_{\widetilde\cT}\big]\nonumber\\
&+\tfrac{N}{2}\mE\big[\norm{\bz-\bz^0}_{\cB+\bC_2}^2-\norm{\bz-\bz^{K}}_{\cB-3\bC_2}^2\big],
\end{align*}
where $\widetilde\cT= \cT-\bL^\Phi-2\bD$. Finally, rearranging the terms in the aforementioned inequality and dropping the negative terms due to the step-size selection lead to the desired result. \qed

\section{Numerical Experiments}\label{numeric}
In this section, we consider a distributed localization problem to test the performance of our proposed algorithm. Given a set of local ellipsoids $\cX_i\triangleq \{x\in[-1,1]^n\mid \norm{A_ix-b_i}\leq \eta_i\}$, for $i\in\cN$, where $A_i\in\reals^{p_i\times n}$, $b_i\in\reals^{p_i}$, and $\eta_i>0$, the goal is to solve the following optimization problem:
%\begin{align}
$\min_{x\in\reals^n} \{\sum_{i\in\cN} f_i(x) %\quad \hbox{s.t.}\quad
\mid x\in\bigcap_{i=1}^N\cX_i\}$, 
%\end{align}
over a network $\cG=(\cN,\cE)$. 
To highlight the benefit of our method, we compare ours with a synchronous distributed primal-dual method (DPDA-S) in \cite{hamedani2021decentralized}.

In this experiments, we set $n=100$, $N=50$, $p_i=50$, and $f_i(x)=\frac{1}{2}\norm{x}^2$ for all $i\in\cN$. We generate a vector $\bar x\in\reals^n$ such that its entries are i.i.d with uniform distribution on $[-1,1]$. For each $i\in\cN$, $A_i$ is generated with a standard Gaussian distribution, $\eta_i$ uniformly at random on $[1,2]$, and $b_i=A_ix+\epsilon_i$ where $\epsilon_i$ is generated with a normal distribution of mean zero and variance of 0.01. Moreover, to generate the network $\cG$, we generated a random small-world network, i.e., we create a cycle over nodes, then we add $N/2$ edges at random with uniform probability. This leads to a connected graph with $|\cE|=75$. 

The results are depicted in Figure \ref{fig1} in terms of suboptimality, infeasibility, and consensus violation versus the number of communications. Note that DPDA-S performs $N$ communications at each iteration while AD-APD performs only one communication per iteration. From Figure \ref{fig1}, we see within the same number of communications AD-APD with asynchronous updates has a better performance than DPDA-S with synchronous updates. \vspace{-5mm}
% \begin{table}[htbp]
% \caption{Table Type Styles}
% \begin{center}
% \begin{tabular}{|c|c|c|c|}
% \hline
% \textbf{Table}&\multicolumn{3}{|c|}{\textbf{Table Column Head}} \\
% \cline{2-4} 
% \textbf{Head} & \textbf{\textit{Table column subhead}}& \textbf{\textit{Subhead}}& \textbf{\textit{Subhead}} \\
% \hline
% copy& More table copy$^{\mathrm{a}}$& &  \\
% \hline
% \multicolumn{4}{l}{$^{\mathrm{a}}$Sample of a Table footnote.}
% \end{tabular}
% \label{tab1}
% \end{center}
% \end{table}
\begin{figure}[htbp]\centering
\subfloat[{suboptimality}]{{\includegraphics[scale=0.2]{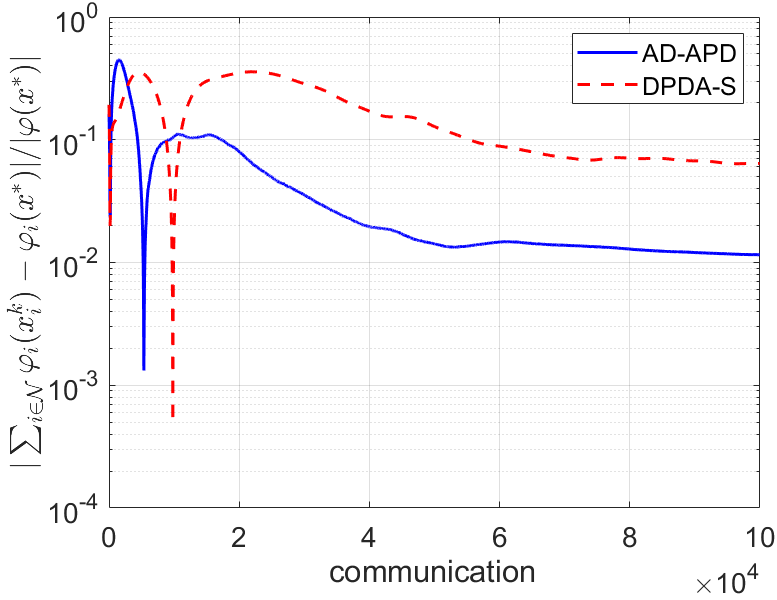}}}
\subfloat[{infeasibility}]{{\includegraphics[scale=0.2]{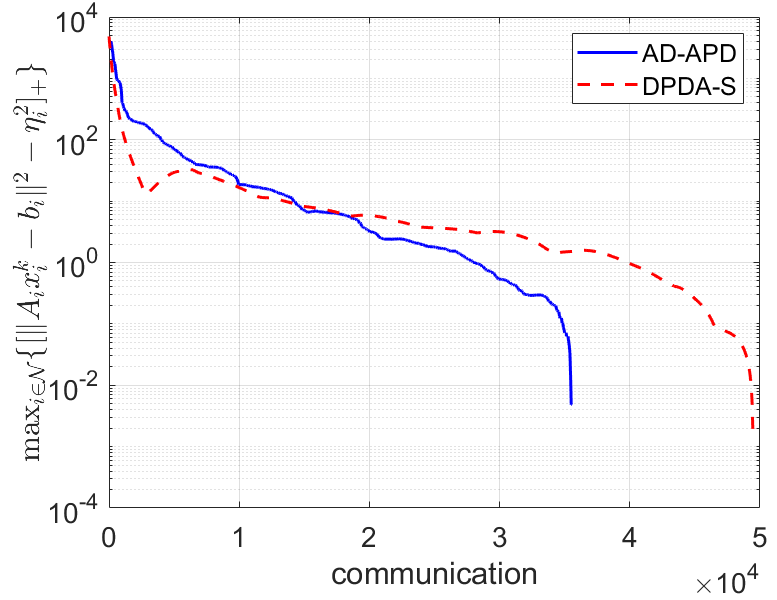}}}\\
\subfloat[{consensus violation}]{{\includegraphics[scale=0.2]{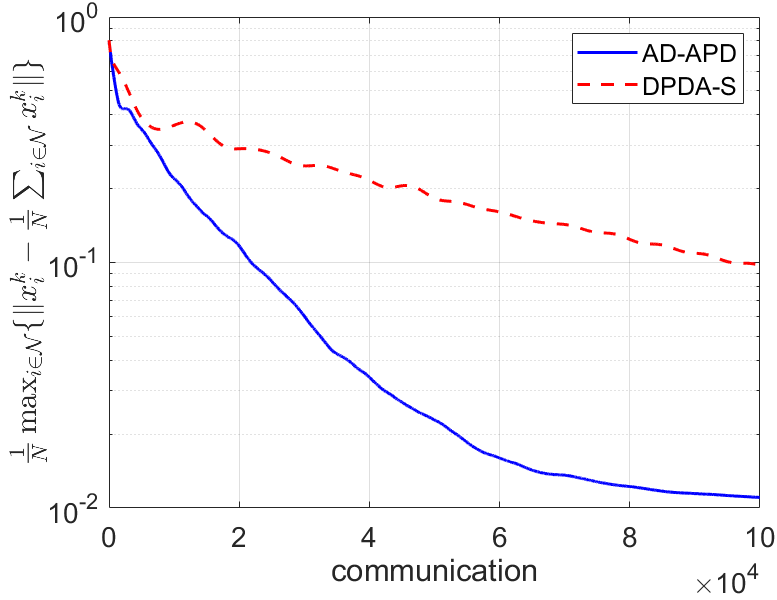}}}
\caption{Comparison of AD-APD and DPDA-S.}%the methods in terms of suboptimality, infeasibility, and consensus violation.}
\label{fig1}
\end{figure}%\vspace{-3mm}
\bibliographystyle{IEEEtran}
\vspace{-1mm}
\bibliography{reference}
\end{document}